\documentclass{article}%
\usepackage{amsmath}
\usepackage{amsfonts}
\usepackage{amssymb}
\usepackage{graphicx}%
\providecommand{\U}[1]{\protect\rule{.1in}{.1in}}
\newtheorem{theorem}{Theorem}

\newtheorem{lemma}[theorem]{Lemma}

\newenvironment{proof}[1][Proof]{\noindent\textbf{#1.} }{\ \rule{0.5em}{0.5em}}
\begin{document}

\begin{center}
\textbf{Characterizing Level-Set Families of Harmonic Functions}%
{\large }

Pisheng Ding
\end{center}

\begin{quote}
\textsc{Abstract. }{\small Families of hypersurfaces that are level-set
families of harmonic functions free of critical points are characterized by a
local differential-geometric condition. Harmonic functions with a specified
level-set family are constructed from geometric data. As a by-product, it is
shown that the evolution of the gradient of a harmonic function along the
gradient flow is determined by the mean curvature of the level sets that the
flow intersects.}
\end{quote}

\section{Introduction}

Harmonic functions on $%
\mathbb{R}
^{n}$ are those whose Laplacian vanishes identically. In this note, by
analyzing certain differential-geometric properties of their level sets, we
give a local characterization of their level-set families.

For harmonic functions of two variables, it is already quite difficult to
characterize their level curves; see, e.g., \cite{F-D-S}. Level hypersurfaces
of harmonic functions of more than two variables are even more intractable
(especially without the complex-analytic tools available in the two-variable
case). This note shows that \textit{families} of level-sets of harmonic
functions free of critical points are somehow easier to characterize. The
difference between an individual curve or hypersurface and a family of them is
that the former is \textquotedblleft static\textquotedblright\ whereas the
latter contains \textquotedblleft kinematic\textquotedblright\ information
that is more readily relatable to harmonicity.

In this Introduction, we state our main result in the two-variable case for
simplicity. The general case will be treated in \S 3.

\begin{theorem}
\label{Thm Main}Let $\Phi:%
\mathbb{R}
\times(-\epsilon,\epsilon)\rightarrow%
\mathbb{R}
^{2}$ be an orientation-preserving $C^{2}$ diffeomorphism onto a domain
$\Omega\subset%
\mathbb{R}
^{2}$; let $\left(  x(\sigma;t),\,y(\sigma;t)\right)  =\Phi(\sigma;t)$. For
each $t\in(-\epsilon,\epsilon)$, define $\gamma_{t}$ to be the curve
$\sigma\mapsto\Phi(\sigma;t)$.\ Let%
\[
\varphi=\frac{\det d\Phi}{\left\Vert \gamma_{t}^{\prime}\right\Vert
}\text{\quad and\quad}\mathbf{N}=\frac{\left(  -\partial y/\partial
\sigma\right)  \mathbf{e}_{1}+\left(  \partial x/\partial\sigma\right)
\mathbf{e}_{2}}{\left\Vert \gamma_{t}^{\prime}\right\Vert }\text{ ;}%
\]
let $s$ be the arc-length parameter (modulo an additive constant) along each
integral curve of $\mathbf{N}$. For $p\in\Omega$, let $\kappa(p)$ be the
signed curvature at $p$ of the curve $\gamma_{t}$ on which $p$ lies (with
$\kappa$ signed in accordance with the normal field $\mathbf{N}$). Then, there
exists a critical-point-free harmonic function $U$ on $\Omega$ with $\left\{
\gamma_{t}\mid t\in(-\epsilon,\epsilon)\right\}  $ being its level-curve
family iff $(\partial\varphi/\partial s)+\kappa\varphi$ is constant on each
curve $\gamma_{t}$, i.e.,%
\begin{equation}
\frac{\partial}{\partial\sigma}\left(  \frac{\partial\varphi}{\partial
s}+\kappa\varphi\right)  \equiv0\text{ .} \label{Eq Intro}%
\end{equation}

\end{theorem}

A few words on notation are in order. The quantity $(\partial\varphi/\partial
s)+\kappa\varphi$ can be construed as a function on both $\Omega$ and $%
\mathbb{R}
\times(-\epsilon,\epsilon)$, via the mapping $\Phi$ between the two domains.
Strictly speaking, for $\partial/\partial\sigma$ to be meaningful, we should
interpret $(\partial\varphi/\partial s)+\kappa\varphi$ as a function on $%
\mathbb{R}
\times(-\epsilon,\epsilon)$. Thus, for $q\in%
\mathbb{R}
\times(-\epsilon,\epsilon)$, $(\kappa\varphi)(q)$ means $\kappa(\Phi
(q))\varphi(q)$, whereas $(\partial\varphi/\partial s)(q)$ is interpreted as
follows: if $\alpha:\left(  -\delta,\delta\right)  \rightarrow\Omega$ is the
(unit-speed) integral curve of $\mathbf{N}$ with $\alpha(0)=\Phi(q)$, then%
\[
\frac{\partial\varphi}{\partial s}(q)=(\varphi\circ\Phi^{-1}\circ
\alpha)^{\prime}(0)=\left.  \frac{d}{ds}\right\vert _{s=0}\varphi(\Phi
^{-1}(\alpha(s)))\text{ .}%
\]

\medskip

\noindent\textit{Remark}. To a $C^{2}$ family of simple closed curves
$\Psi:S^{1}\times(-\epsilon,\epsilon)\rightarrow%
\mathbb{R}
^{2}$, Theorem \ref{Thm Main} can be applied in either of the following two ways.

First, we may use the covering map $\pi:%
\mathbb{R}
\rightarrow S^{1}$, $\sigma\mapsto(\cos\sigma,\sin\sigma)$, to define%
\[
\Phi:%
\mathbb{R}
\times(-\epsilon,\epsilon)\overset{\pi\times\operatorname*{Id}%
}{\longrightarrow}S^{1}\times(-\epsilon,\epsilon)\overset{\Psi}{\rightarrow}%
\mathbb{R}
^{2}\text{ ,}%
\]
which is a local diffeomorphism. With this $\Phi$, Theorem \ref{Thm Main} applies.

In the other approach, which would prove suitable for the more general case
treated in \S 3, we cover $S^{1}$ with an atlas of two local charts $\phi_{i}:%
\mathbb{R}
\rightarrow\mathcal{O}_{i}\hookrightarrow S^{1}$ and then define%
\[
\Phi_{i}:%
\mathbb{R}
\times(-\epsilon,\epsilon)\overset{\phi_{i}\times\operatorname*{Id}%
}{\longrightarrow}\mathcal{O}_{i}\times(-\epsilon,\epsilon)\hookrightarrow
S^{1}\times(-\epsilon,\epsilon)\overset{\Psi}{\rightarrow}%
\mathbb{R}
^{2}\text{ .}%
\]
Theorem \ref{Thm Main} then implies that the given family of simple closed
curves is the level-curve family of a harmonic function iff both $\Phi_{i}$
meet Condition (\ref{Eq Intro}).

\smallskip

Theorem \ref{Thm Main} will be proved in \S 2.2, following an exposition of
some preliminaries in \S 2.1. In \S 3, we generalize this result to the
many-variable case, wherein the mean curvature $H$ of level hypersurfaces
figures in place of $\kappa$ in a condition parallel to (\ref{Eq Intro}) that
characterizes level-set families of harmonic functions. When the condition is
met by a given family, we construct in \S 4 a harmonic function $U$ whose
level-set family is the given one.

Related to the theme of this note is another fundamental observation
concerning how the geometry of the level sets of a harmonic function determine
its gradient. We now state this result.

\begin{theorem}
\label{Thm Gradient}Let $U$ be a critical-point-free harmonic function on a
domain $\Omega\subset%
\mathbb{R}
^{n}$. Let $\alpha:\left(  -\delta,\delta\right)  \rightarrow\Omega$ be an
integral curve of the vector field $\nabla U/\left\Vert \nabla U\right\Vert $.
Then, for $s\in\left(  -\delta,\delta\right)  $,%
\[
\left\Vert \nabla U(\alpha(s))\right\Vert =\left\Vert \nabla U(\alpha
(0))\right\Vert \exp\left(  (n-1)\int_{0}^{s}H(\alpha(\xi))\,d\xi\right)
\text{ ,}%
\]
where, for any $p\in\Omega$, $H(p)$ is the mean curvature at $p$ of the
level-$U(p)$ hypersurface oriented by the normal field $\nabla U/\left\Vert
\nabla U\right\Vert $.
\end{theorem}

Note that $\alpha$ by definition has unit speed and therefore is the
arc-length parametrization of a gradient flow.

This theorem will be proved in \S 3.

\section{The Two-variable Case}

We base our proof for Theorem \ref{Thm Main} on a characterization of
harmonicity of a $C^{2}$ function in terms of curvature of its level sets and
its directional derivatives, which we first turn to.

\subsection{Characterizing Harmonicity in Terms of Curvature of Level Curves}

Let $f$ be a $C^{2}$ function on a domain $\Omega\subset%
\mathbb{R}
^{2}$ with no critical points. For $p\in\Omega$ and a unit vector
$\mathbf{v}\in T_{p}%
\mathbb{R}
^{2}$, denote by $D_{\mathbf{v}}f(p)$ and $D_{\mathbf{v}}^{2}f(p)$ the first
and second directional derivatives of $f$ at $p$ along $\mathbf{v}$. Denote by
$Q_{p}:T_{p}%
\mathbb{R}
^{2}\times T_{p}%
\mathbb{R}
^{2}\rightarrow%
\mathbb{R}
$ the Hessian quadratic form of $f$ at $p$; recall that $D_{\mathbf{v}}%
^{2}f(p)=Q_{p}(\mathbf{v},\mathbf{v})$.

Let $C$ be a level curve of $f$. Install on $C$ the unit normal field
$\mathbf{N}:=\nabla f/\left\Vert \nabla f\right\Vert $; for the unit tangent
field $\mathbf{T}$, let it be such that the frame $(\mathbf{T},\mathbf{N})$ is
positively-oriented (but the opposite choice will do as well). The
\textit{signed}\ \textit{curvature} $\kappa$ of the curve $C$ at each point
thereon is defined by the equation $d\mathbf{T}/d\sigma=\kappa\mathbf{N}$,
where $\sigma$ is the arc-length parameter along $C$ (with its increasing
direction induced by $\mathbf{T}$); note that the sign of $\kappa$ depends on
the choice we make of $\mathbf{N}$, but not of $\mathbf{T}$. So defined,
$\kappa$ is a scalar field on $\Omega$.

\begin{lemma}
\label{Lemma Curvature of Level Curves}Assume the preceding hypothesis. For
$p\in\Omega$,%
\begin{equation}
\kappa(p)=-\frac{D_{\mathbf{T}}^{2}f(p)}{D_{\mathbf{N}}f(p)}=\frac
{D_{\mathbf{N}}^{2}f(p)-\Delta f(p)}{D_{\mathbf{N}}f(p)}\text{ .}
\label{Eq Curvature Level Curves, General}%
\end{equation}
Consequently, $f$ is harmonic iff%
\begin{equation}
\kappa\equiv\frac{D_{\mathbf{N}}^{2}f}{D_{\mathbf{N}}f}\text{\quad on }%
\Omega\,\text{.} \label{Eq Curvature Level Curves, Harmonic}%
\end{equation}

\end{lemma}

\begin{proof}
Let $\gamma$ be the unit-speed parametrization of an arc on the level-$f(p)$
curve, with $\gamma(0)=p$ and $\gamma^{\prime}(0)=\mathbf{T}(p)$. By
definition, $\gamma^{\prime\prime}(0)=\kappa(p)\mathbf{N}(p)$. For all $t$,
$\left\langle \nabla f(\gamma(t)),\,\gamma^{\prime}(t)\right\rangle =0$; hence%
\[
\left\langle \frac{d}{dt}\nabla f(\gamma(t)),\,\gamma^{\prime}(t)\right\rangle
+\left\langle \nabla f(\gamma(t)),\,\gamma^{\prime\prime}(t)\right\rangle
=0\text{ .}%
\]
Note that%
\[
\left\langle \left.  \frac{d}{dt}\right\vert _{t=0}\nabla f(\gamma
(t),\,\gamma^{\prime}(0)\right\rangle =D_{\mathbf{T}}^{2}f(p)\text{ ,}%
\]
whereas%
\[
\left\langle \nabla f(\gamma(0)),\,\gamma^{\prime\prime}(0)\right\rangle
=\kappa(p)\left\Vert \nabla f(p)\right\Vert \text{ .}%
\]
Now (\ref{Eq Curvature Level Curves, General}) follows by noting that
$D_{\mathbf{N}(p)}f(p)=\left\Vert \nabla f(p)\right\Vert $ and that%
\[
D_{\mathbf{T}}^{2}f(p)+D_{\mathbf{N}}^{2}f(p)=\operatorname*{Tr}Q_{p}=\Delta
f(p)\text{ .}%
\]

\end{proof}

We are now ready to establish Theorem \ref{Thm Main}.

\subsection{Proof of Theorem \ref{Thm Main}}

Let $J$ denote the interval $(-\epsilon,\epsilon)$. Recall from \S 1 that
$\Phi:%
\mathbb{R}
\times J\rightarrow%
\mathbb{R}
^{2}$, with $\Phi(\sigma;t)=\left(  x(\sigma;t),\,y(\sigma;t)\right)  $, is an
orientation-preserving $C^{2}$ diffeomorphism onto a domain $\Omega$, that
$\gamma_{t}$ is the curve $\sigma\mapsto\Phi(\sigma,t)$, and that%
\begin{equation}
\varphi=\frac{\det d\Phi}{\left\Vert \gamma_{t}^{\prime}\right\Vert
}\text{\quad and\quad}\mathbf{N}=\frac{\left(  -\partial y/\partial
\sigma\right)  \mathbf{e}_{1}+\left(  \partial x/\partial\sigma\right)
\mathbf{e}_{2}}{\left\Vert \gamma_{t}^{\prime}\right\Vert }\text{ .}
\label{Eq Definition of N & Phi}%
\end{equation}
(Note that $\varphi>0$.) Theorem \ref{Thm Main} asserts that there exists a
harmonic function $U$ on $\Omega$ with $\left\{  \gamma_{t}\right\}  _{t\in
J}$ being its level-curve family iff%
\begin{equation}
\frac{\partial}{\partial\sigma}\left(  \frac{\partial\varphi}{\partial
s}+\kappa\varphi\right)  \equiv0\text{ ,} \label{Eq Master Eq}%
\end{equation}
where $s$ is the arc-length parameter along each integral curve of
$\mathbf{N}$ and $\kappa$ is the signed curvature of $\gamma_{t}$ (signed in
accordance with the normal field $\mathbf{N}$).

We now set out to prove Theorem \ref{Thm Main}.

Let $t:\Omega\rightarrow J$ denote the second component of $\Phi^{-1}$. (For
$p\in\Omega$, $t(p)$ is characterized by the condition that $p$ is on the
curve $\gamma_{t(p)}$.) By the inverse function theorem,%
\[
\nabla t=\frac{1}{\varphi}\mathbf{N}\text{ .}%
\]

Any function having the $\gamma_{t}$'s as its level curves evidently depends
solely on $t$. Therefore, consider a function $U:\Omega\rightarrow%
\mathbb{R}
$ of the form $u\circ t$ where $u:J\rightarrow%
\mathbb{R}
$ is a $C^{2}$ function free of critical points. Then,%
\[
\nabla U=u^{\prime}\cdot\nabla t=\frac{u^{\prime}}{\varphi}\mathbf{N}\text{ .}%
\]
Clearly, $U$ is free of critical point on $\Omega$. Note that $u^{\prime}$ on
$J$ is either strictly positive or strictly negative (by Darboux's theorem).
Without loss of generality, assume that $u^{\prime}>0$, in which case $\nabla
U/\left\Vert \nabla U\right\Vert =\mathbf{N}$ and the curvature of level
curves of $U$ as defined in \S 2.1\ has the same sign as the curvature
$\kappa$ on $\Omega$ introduced in \S 1.

Now that $U=u\circ t$ is a critical-point-free $C^{2}$ function on $\Omega$
with $\left\{  \gamma_{t}\right\}  _{t\in J}$ being its level-curve family, we
seek a necessary and sufficient condition for $U$ to be harmonic.

With $s$ denoting the arc-length parameter along the integral curves of
$\mathbf{N}$ (which are also gradient flows of $t$), we compute $D_{\mathbf{N}%
}U$ and $D_{\mathbf{N}}^{2}U$. First,%
\[
D_{\mathbf{N}}U=\frac{\partial}{\partial s}(u\circ t)=u^{\prime}%
(t)\frac{\partial t}{\partial s}=u^{\prime}(t)\left\Vert \nabla t\right\Vert
=\frac{u^{\prime}(t)}{\varphi}%
\]
(or more simply:\ $D_{\mathbf{N}}U=\left\langle \nabla U,\mathbf{N}%
\right\rangle =u^{\prime}(t)/\varphi$\thinspace); also%
\[
D_{\mathbf{N}}^{2}U=\frac{\partial^{2}}{\partial s^{2}}(u\circ t)=\frac
{\partial}{\partial s}\frac{u^{\prime}(t)}{\varphi}=\frac{u^{\prime\prime
}(t)\left\Vert \nabla t\right\Vert \varphi-u^{\prime}(t)\frac{\partial\varphi
}{\partial s}}{\varphi^{2}}=\frac{u^{\prime\prime}(t)-u^{\prime}%
(t)\frac{\partial\varphi}{\partial s}}{\varphi^{2}}\text{ .}%
\]
Hence, the condition that $D_{\mathbf{N}}^{2}U=\kappa D_{\mathbf{N}}U$ is
equivalent to the condition that $\left(  u^{\prime\prime}(t)-u^{\prime
}(t)\frac{\partial\varphi}{\partial s}\right)  /\varphi^{2}=\kappa u^{\prime
}(t)/\varphi$, i.e., the condition that%
\begin{equation}
\frac{\partial\varphi}{\partial s}+\kappa\varphi=\frac{u^{\prime\prime}%
(t)}{u^{\prime}(t)}\text{ .} \label{Eq Condition}%
\end{equation}

If $U$ is harmonic, then, by Lemma \ref{Lemma Curvature of Level Curves},
$D_{\mathbf{N}}^{2}U=\kappa D_{\mathbf{N}}U$ and hence (\ref{Eq Condition})
holds. In equation (\ref{Eq Condition}), the right side of the equality is
constant on each curve $\gamma_{t}$ and hence independent of the parameter
$\sigma$ that parametrizes $\gamma_{t}$\thinspace; so must the left side! This
proves the necessity of Condition (\ref{Eq Master Eq}) for $U$ to be harmonic.

Conversely, if Condition (\ref{Eq Master Eq}) holds, then, $\left(
\partial\varphi/\partial s\right)  +\kappa\varphi$ depends only on $t$.
Viewing $\left(  \partial\varphi/\partial s\right)  +\kappa\varphi$ as a
function of $t$, we can find a function $u$ of $t$ such that%
\[
\frac{d}{dt}\log u^{\prime}(t)=\frac{\partial\varphi}{\partial s}%
+\kappa\varphi\text{ .}%
\]
For any such $u$, (\ref{Eq Condition}) holds and hence the composite function
$U:=u\circ t$ satisfies the condition $D_{\mathbf{N}}^{2}U=\kappa
D_{\mathbf{N}}U$ on $\Omega$, which, by Lemma
\ref{Lemma Curvature of Level Curves}, implies that $U$ is harmonic.

\section{Extension to Higher Dimensions}

Turning to the $n$-variable case, we give an immediate extension of Theorem
\ref{Thm Main}. We begin with a family of hypersurfaces in $%
\mathbb{R}
^{n}$ and seek a condition that is necessary and sufficient for the existence
of a harmonic function whose level-set family is the given one. For
simplicity, suppose that the hypersurfaces in the family are all diffeomorphic
to $%
\mathbb{R}
^{n-1}$. As we shall remark following Theorem \ref{Thm High Dimensions}, this
is not a severe restriction.

Let $J$ denote the interval $(-\epsilon,\epsilon)$. Suppose that $\Phi:%
\mathbb{R}
^{n-1}\times J\rightarrow%
\mathbb{R}
^{n}$, with $\Phi(x;t)=\left(  y_{1}(x;t),\,y_{2}(x;t),\cdots,y_{n}%
(x;t)\right)  $, is an orientation-preserving $C^{2}$ diffeomorphism onto a
domain $\Omega$ in $%
\mathbb{R}
^{n}$. For each $t\in J$, define $\Gamma_{t}:%
\mathbb{R}
^{n-1}\rightarrow%
\mathbb{R}
^{n}$ to be the parametrized hypersurface $x\mapsto\Phi(x;t)$. Let
$\mathbf{n}$ be the Hodge dual of $(\partial\Phi/\partial x_{1})\wedge
(\partial\Phi/\partial x_{2})\wedge\cdots\wedge(\partial\Phi/\partial
x_{n-1})$, i.e.,%
\[
\mathbf{n}=(-1)^{n-1}\det\left[
\begin{array}
[c]{cccc}%
\mathbf{e}_{1} & \frac{\partial y_{1}}{\partial x_{1}} & \cdots &
\frac{\partial y_{1}}{\partial x_{n-1}}\\
\mathbf{e}_{2} & \frac{\partial y_{2}}{\partial x_{1}} & \cdots &
\frac{\partial y_{2}}{\partial x_{n-1}}\\
\vdots & \vdots & \vdots & \vdots\\
\mathbf{e}_{n} & \frac{\partial y_{n}}{\partial x_{1}} & \cdots &
\frac{\partial y_{n}}{\partial x_{n-1}}%
\end{array}
\right]  \text{ .}%
\]
Let%
\[
\varphi=\frac{\det d\Phi}{\left\Vert \mathbf{n}\right\Vert }\text{\quad
and\quad}\mathbf{N}=\frac{\mathbf{n}}{\left\Vert \mathbf{n}\right\Vert
}\text{.}%
\]
Orient each $\Gamma_{t}$ by the normal field $\mathbf{N}$. For $p\in\Omega$,
let $H(p)$ be the mean curvature at $p$ of the oriented hypersurface
$\Gamma_{t}$ on which $p$ lies. Let $s$ denote the arc-length parameter
(modulo an additive constant) along the integral curves of $\mathbf{N}$.

\begin{theorem}
\label{Thm High Dimensions}Assume the preceding hypotheses and notations.
There exists a critical-point-free harmonic function $U$ whose family of level
sets is $\left\{  \Gamma_{t}\right\}  $ iff $(\partial\varphi/\partial
s)+(n-1)H\varphi$ is constant on each hypersurface $\Gamma_{t}$, i.e.,%
\begin{equation}
\frac{\partial}{\partial x_{i}}\left(  \frac{\partial\varphi}{\partial
s}+(n-1)H\varphi\right)  =0\text{\quad for }i\in\{1,2,\cdots,(n-1)\}\text{.}
\label{Eq Master Eq - General}%
\end{equation}

\end{theorem}

\medskip

\noindent\textit{Remark}. To a $C^{2}$ family of hypersurfaces $\Psi
:M^{n-1}\times(-\epsilon,\epsilon)\rightarrow%
\mathbb{R}
^{n}$, where $M^{n-1}$ is a connected $C^{2}$ $(n-1)$-manifold, Theorem
\ref{Thm High Dimensions} can be applied as follows.

Take a $C^{2}$ atlas of local charts $\phi_{i}:%
\mathbb{R}
^{n-1}\rightarrow\mathcal{O}_{i}\hookrightarrow M^{n-1}$. Define%
\[
\Phi_{i}:%
\mathbb{R}
^{n-1}\times(-\epsilon,\epsilon)\overset{\phi_{i}\times\operatorname*{Id}%
}{\longrightarrow}\mathcal{O}_{i}\times(-\epsilon,\epsilon)\hookrightarrow
M^{n-1}\times(-\epsilon,\epsilon)\overset{\Psi}{\rightarrow}%
\mathbb{R}
^{n}\text{ .}%
\]
Theorem \ref{Thm High Dimensions} then implies that the given family of
hypersurfaces is the level-set family of a critical-point-free harmonic
function iff $\Phi_{i}$ meet Condition (\ref{Eq Master Eq - General}) for all
$i$.

\medskip

All that is needed for the proof is a characterization of harmonicity that
extends Lemma \ref{Lemma Curvature of Level Curves}. We briefly recall the
mean curvature of a one-codimensional $C^{2}$ orientable submanifold $M$ of $%
\mathbb{R}
^{n}$. Let $M$ be oriented by a unit normal field $\mathbf{N}$. At a point
$p\in M$, any unit tangent vector $\mathbf{v}\in T_{p}M$ and the unit normal
$\mathbf{N}(p)$ span a plane $\Pi_{\mathbf{v}}$; the intersection curve
$M\cap\Pi_{\mathbf{v}}$ is a so-called normal section at $p$ and has signed
curvature $\kappa_{p}(\mathbf{v})=-D_{\mathbf{v}}^{2}f(p)/\left\Vert \nabla
f(p)\right\Vert $, as in (\ref{Eq Curvature Level Curves, General}). The mean
curvature $H(p)$ of $M$ at $p$ is simply the average of $\kappa_{p}%
(\mathbf{v})$ as $\mathbf{v}$ ranges over the unit $(n-2)$-sphere in the
$(n-1)$-dimensional $T_{p}M$. The following result, like Lemma
\ref{Lemma Curvature of Level Curves}, plays a key role in establishing
Theorem \ref{Thm High Dimensions}.

\begin{lemma}
\label{Lemma Curvature of Level Surfaces}For a $C^{2}$ real-valued function
$f$ on an open subset $\Omega$ of $%
\mathbb{R}
^{n}$ without critical points, let $\mathbf{N}=\nabla f/\left\Vert \nabla
f\right\Vert $ and let $H(p)$ be the mean curvature at $p\in\Omega$ of the
level-$f(p)$ hypersurface of $f$. Then,%
\begin{equation}
H(p)=\frac{1}{n-1}\frac{D_{\mathbf{N}}^{2}f(p)-\Delta f\left(  p\right)
}{D_{\mathbf{N}}f(p)}\text{ .} \label{Eq Mean Curvature}%
\end{equation}
Consequently, $f$ is harmonic iff%
\begin{equation}
H\equiv\frac{1}{n-1}\frac{D_{\mathbf{N}}^{2}f}{D_{\mathbf{N}}f}\text{\quad on
}\Omega\,\text{.} \label{Eq Mean Curvature, Harmonic}%
\end{equation}

\end{lemma}

We omit the simple proof of (\ref{Eq Mean Curvature}). The proof of Theorem
\ref{Thm High Dimensions} parallels that of Theorem \ref{Thm Main}; it
suffices to note that equation (\ref{Eq Condition}), key to Theorem
\ref{Thm Main}, now takes the form%
\begin{equation}
\frac{\partial\varphi}{\partial s}+(n-1)H\varphi=\frac{u^{\prime\prime}%
(t)}{u^{\prime}(t)}\text{ .} \label{Eq Condition General}%
\end{equation}

From Lemma \ref{Lemma Curvature of Level Surfaces}, we deduce Theorem
\ref{Thm Gradient} stated in \S 1.

\medskip

\noindent\textbf{Proof of Theorem \ref{Thm Gradient}}. Let $U$ be a
critical-point-free harmonic function on a domain $\Omega\subset%
\mathbb{R}
^{n}$. Let $\mathbf{N}=\nabla U/\left\Vert \nabla U\right\Vert $. Let
$s\mapsto\alpha(s)$ be the unit-speed gradient flow originating from $p_{0}$;
$\alpha$ is such that%
\[
\alpha(0)=p_{0}\text{\quad and\quad}\alpha^{\prime}(s)=\mathbf{N}%
(\alpha(s))\text{ .}%
\]

Consider $g:=U\circ\alpha$. It is elementary that%
\[
g^{\prime}(s)=D_{\mathbf{N}}U(\alpha(s))=\left\Vert \nabla U(\alpha
(s))\right\Vert \text{\quad and\quad}g^{\prime\prime}(s)=D_{\mathbf{N}}%
^{2}U(\alpha(s))\text{ .}%
\]
By (\ref{Eq Mean Curvature, Harmonic}),%
\begin{equation}
(n-1)H(\alpha(s))=\frac{g^{\prime\prime}(s)}{g^{\prime}(s)}=\frac{d}{ds}\log
g^{\prime}(s)=\frac{d}{ds}\log\left\Vert \nabla U(\alpha(s))\right\Vert \text{
.}\nonumber
\end{equation}
Integrating both sides yields that%
\[
\left\Vert \nabla U(\alpha(s))\right\Vert =\left\Vert \nabla U(p_{0}%
)\right\Vert \exp\left(  (n-1)\int_{0}^{s}H(\alpha(\xi))\,d\xi\right)  \text{
,}%
\]
proving Theorem \ref{Thm Gradient}.\qquad\qquad$\blacksquare$

\medskip

\noindent\textit{Remark}. When $n=2$ in Theorem \ref{Thm Gradient}, the mean
curvature $H$ of level hypersurfaces is replaced by the signed curvature
$\kappa$ of level curves, resulting in%
\[
\left\Vert \nabla U(\alpha(s))\right\Vert =\left\Vert \nabla U(\alpha
(0))\right\Vert \exp\left(  \int_{0}^{s}\kappa(\alpha(\xi))\,d\xi\right)
\text{ .}%
\]
This formula is suggested in \cite{J-R} by a purely complex-analytic argument,
as a harmonic function on a planar domain is locally the real part of a
complex analytic function. However, that argument does not apply when $n>2$.

\section{Constructing a Harmonic Function from Its Family of Level Sets}

Suppose that the hypersurface family $\left\{  \Gamma_{t}\right\}  $ defined
in Theorem \ref{Thm High Dimensions} satisfies (\ref{Eq Master Eq - General})
and therefore is the level-set family of a harmonic function $U$. We seek to
determine $U$ from $\left\{  \Gamma_{t}\right\}  $. As we shall see, the final
outcome will provide a link between Theorems \ref{Thm High Dimensions} and
\ref{Thm Gradient}, the two main results of this note.

As in \S 2.2, for $y\in\Omega$, suppose that $U(y)=u(t(y))$ with $u^{\prime
}>0$ on $(-\epsilon,\epsilon)$. Recall equation (\ref{Eq Condition General}),
which is equivalent to $U$ being harmonic:%
\[
\frac{u^{\prime\prime}(t)}{u^{\prime}(t)}=\frac{\partial\varphi}{\partial
s}+(n-1)H\varphi\text{ .}%
\]
Integrating both sides yields%
\begin{equation}
\log\frac{u^{\prime}(t)}{u^{\prime}(0)}=\int_{0}^{t}\left.  \left(
\frac{\partial\varphi}{\partial s}+(n-1)H\varphi\right)  \right\vert
_{(x(\tau),\tau)}d\tau\text{ ,} \label{Eq log(u')}%
\end{equation}
where $x(\tau)\in%
\mathbb{R}
^{n-1}$ is so chosen that the curve%
\[
\ell:(-\epsilon,\epsilon)\rightarrow\Omega,\text{\ }\tau\mapsto\Phi
(x(\tau);\tau)
\]
is a parametrization of the integral curve of $\mathbf{N}$ originating at
$\Phi(\mathbf{0};0)$. (Due to (\ref{Eq Master Eq - General}), the expression
$(\partial\varphi/\partial s)+(n-1)H\varphi$ is constant in $x$; therefore,
for its evaluation at $\left(  x;\tau\right)  $, we are free to choose $x$.
Our particular choice facilitates further evaluation of the integral in
(\ref{Eq log(u')}).) As $\nabla t=(1/\varphi)\mathbf{N}$, the curve $\ell$ is
also a parametrization of the gradient flow of $t$ originating at
$\Phi(\mathbf{0};0)$; with $s$ being the arc-length parameter along $\ell$,
$dt/ds=\left\Vert \nabla t\right\Vert =1/\varphi$. Thus,%
\[
\int_{0}^{t}\frac{\partial\varphi}{\partial s}d\tau=\int_{\ell}\frac
{\partial\varphi}{\partial s}\frac{1}{\varphi}ds=\log\frac{\varphi
(x(t);t)}{\varphi(\mathbf{0};0)}\text{ ,}%
\]
which allows us to conclude from (\ref{Eq log(u')}) that%
\[
u^{\prime}(t)=\frac{u^{\prime}(0)}{\varphi(\mathbf{0};0)}\varphi
(x(t);t)\exp\int_{0}^{t}(n-1)H(\ell(\tau))\varphi(x(\tau);\tau)\,d\tau\text{
.}%
\]
Integrating, we have%
\[
u(T)=u(0)+\frac{u^{\prime}(0)}{\varphi(\mathbf{0};0)}\int_{0}^{T}\left(
\varphi(x(t);t)\exp\int_{0}^{t}(n-1)H(\ell(\tau))\varphi(x(\tau);\tau
)\,d\tau\right)  dt\text{ .}%
\]
Substituting $ds$ for $\varphi d\tau$ and $\varphi dt$ in the two integrals,
the above formula can written in terms of line integrals along $\ell$ with
respect to arc length $s$:%
\begin{equation}
U(\Phi(\mathbf{0};T))=U(\Phi(\mathbf{0};0))+\left\Vert \nabla U(\Phi
(\mathbf{0};0))\right\Vert \int_{\left.  \ell\right\vert _{[0.T]}}e^{(n-1)I}ds
\label{Eq Formula for U}%
\end{equation}
where%
\[
I(\ell(t))=\int_{\left.  \ell\right\vert _{[0.t]}}H\,ds\text{ .}%
\]

\medskip

\noindent\textit{Remark}.\ Formula (\ref{Eq Formula for U}), derived from
Theorem \ref{Thm High Dimensions}, provides a satisfying and corroborating
link between Theorems \ref{Thm High Dimensions} and \ref{Thm Gradient}. By
Theorem \ref{Thm Gradient},%
\[
\left\Vert \nabla U(\Phi(\mathbf{0};0))\right\Vert e^{(n-1)I(\ell
(t))}=\left\Vert \nabla U(\ell(t))\right\Vert \text{ ;}%
\]
with this, Formula (\ref{Eq Formula for U}) then reads%
\[
U(\Phi(\mathbf{0};T))=U(\Phi(\mathbf{0};0))+\int_{\left.  \ell\right\vert
_{[0.T]}}\left\Vert \nabla U\right\Vert ds\text{ ,}%
\]
which is a statement of the fundamental theorem of calculus.

\end{document}